\documentclass[11pt, epsfig]{article}
\usepackage{epsfig, amsmath, amssymb, amsthm, times}
\usepackage{graphicx}

\newcommand{\R}{{\mathbb R}}
\newcommand{\C}{{\mathbb C}}
\newcommand{\N}{{\mathbb N}}
\newcommand{\Z}{{\mathbb Z}}

\usepackage[mathcal]{eucal}

\newtheorem {thm}{Theorem}[section]
\newtheorem {lem}[thm]{Lemma}

\newtheorem {cor}[thm]{Corollary}

 \newtheorem {df}[thm]{Definition}
 \newtheorem{rem}[thm]{Remark}
 \newtheorem{ex}[thm]{Example}
\def\E{\operatorname{\mathbb E}}
\def\P{\operatorname{\mathbb P}}
\def\SS{{\mathbb S}}

\def\lbl{\label}
\def\be{\begin{equation}}
\def\ee{\end{equation}}
\def\p{\partial}

\newcommand{\1}{{i\mkern1mu}}

\title{The mean field analysis for the Kuramoto model on graphs I.\\
The mean field equation and transition point formulas}
\author{Hayato Chiba\thanks{Institute of Mathematics for Industry, 
Kyushu University / JST PRESTO, Fukuoka,
819-0395, Japan, {\tt chiba@imi.kyushu-u.ac.jp}}\;
and Georgi S. Medvedev\thanks{Department of Mathematics, 
Drexel University, 3141 Chestnut Street, Philadelphia, PA 19104,
{\tt medvedev@drexel.edu}} 
}
\begin{document}
\maketitle

\begin{abstract}
In his classical work on synchronization, Kuramoto derived the formula
for the critical value of the coupling strength corresponding to the transition
to synchrony in large ensembles of all-to-all coupled phase oscillators with
randomly distributed intrinsic frequencies.  We extend the Kuramoto's result 
to a large class of coupled systems on convergent families of deterministic and
random graphs. Specifically, we identify the critical values of the coupling strength
(transition points), between which  the incoherent state is linearly stable and 
 is unstable otherwise. We show that the transition points depend on
 the largest positive or/and smallest negative eigenvalue(s) of the  kernel 
operator defined by  the graph limit.
This reveals the precise mechanism, by which the network topology
controls transition to synchrony in the Kuramoto model on graphs. To illustrate
the analysis with concrete examples, we derive the transition point formula for
the coupled systems on Erd\H{o}s-R{\' e}nyi, small-world, and $k$-nearest-neighbor
families of graphs. 
As a result of independent interest,
we provide a rigorous justification for the mean field limit for the Kuramoto model
on graphs. The latter is used in the derivation of the transition point formulas.
\end{abstract}

\section{Introduction}
\setcounter{equation}{0}
Synchronization of coupled oscillators is a classical problem of nonlinear
science with diverse applications in science and engineering \cite{Blekh71, Str-Sync}. 
Physical and technological applications of synchronization include
power, sensor, and communication networks \cite{DorBul12}, 
mobile agents \cite{RenBeard05}, electrical circuits \cite{AVR86}, 
coupled lasers \cite{KLNAL95}, and Josephson junctions \cite{WatStr94}, to name a few. 
In biological and social sciences, synchronization is studied in the context 
of flocking, opinion dynamics, and voting \cite{HenOls16, MotTad14}. 
Synchronization plays a prominent role in physiology and in neurophysiology, in particular.
It is important in the information processing in the brain \cite{Singer93} and 
in the mechanisms of several severe neurodegenerative diseases such as epilepsy 
\cite{TraWhi10} and Parkinson’s Disease \cite{LevHut00}. This list can be continued.

Identifying common principles underlying synchronization in such diverse models is a 
challenging task. In seventies Kuramoto found an elegant approach to this problem. Motivated
by problems in statistical physics and biology, he reduced a system of weakly coupled limit
cycle oscillators to the system of equations for the phase variables only\footnote{For related 
reductions predating Kuramoto's work, see \cite{Malkin49, Malkin56, Blekh71} and the 
discussion before 
Theorem~9.2 in \cite{HopIzh97}.}. The resultant equation is called the Kuramoto model 
(KM) \cite{Kur75}. Kuramoto's method applies directly to a broad class of models in 
natural science. Moreover, it provides a paradigm for studying synchronization. 
The analysis of the KM revealed one of the most striking results of the theory of 
synchronization. For a system of coupled oscillators with randomly distributed
intrinsic frequencies, Kuramoto identified the critical value of the coupling strength, 
at which the gradual buildup of coherence begins. 
He introduced the order parameter, which describes the degree of coherence 
in a coupled system. Using the order parameter, Kuramoto predicted the bifurcation
marking the onset of synchronization. 

Kuramoto's analysis, while not mathematically
rigorous, is based on the correct intuition for the transition to synchronization. 
His discovery initiated a line of fine 
research (see \cite{StrMir91, StrMir92, Str00, Chi15a} and references therein).
It was shown in \cite{StrMir91, StrMir92} that the onset of synchronization corresponds to 
the loss of stability of the incoherent state, a steady state solution of the mean field equation.  
The latter
is a nonlinear hyperbolic partial differential equation for the probability density function
describing the distribution of phases on the unit circle at a given time. The bifurcation 
analysis of the mean field equation is complicated by the presence of the continuous spectrum
of the linearized problem on the imaginary axis. To overcome this problem, in \cite{Chi15a}
Chiba developed an analytical method, which uses the theory of generalized functions 
and rigged Hilbert spaces \cite{GelVilv4}. 

The Kuramoto's result and the analysis in 
\cite{StrMir91, StrMir92,Chi15a} deal with all-to-all coupled systems.
Real world applications feature complex and often random connectivity patterns \cite{PG16}. 
The goal of our work is to describe the onset of synchronization
in the KM on graphs. To this end, we follow the approach in \cite{Med14a, Med14b}. 
Specifically, we consider the KM on convergent families of simple and weighted graphs,
for which  we derive and rigorously justify the mean field limit. Our framework covers 
many random graphs widely used in applications, including Erd\H{o}s-R{\' e}nyi and
small-world graphs. With the mean field equation in hand, we derive the transition point
formulas for the critical values of the coupling strength, where the incoherent state 
loses stability. Thus, we identify the region of linear stability of the incoherent state.
In the follow-up work, we will show that in this region the incoherent state is asymptotically
stable (albeit in weak topology), analyze the bifurcations at the transition points, and 
develop the center manifold reduction.

The original KM with all-to-all coupling and random intrinsic frequencies has the 
following form:
\be\lbl{classKM}
\dot\theta_i=\omega_i+{K\over n} \sum_{j=1}^n \sin(\theta_j-\theta_i).
\ee
Here, $\theta_i:\R\to\SS:=\R/2\pi\Z,\, i\in [n]:=\{1,2,\dots, n\}$ is 
the phase of the oscillator~$i$, whose 
intrinsic frequency $\omega_i$ is drawn from the probability distribution with density 
$g(\omega)$, $n$ is the number of oscillators, and $K$ is the strength of coupling. The sum on the 
right-hand side of \eqref{classKM} describes the interactions of the oscillators in the 
network. The goal is to describe the distribution of $\theta_i(t), \; i\in [n],$ for large times and $n\gg 1.$

Since the intrinsic frequencies are random, for small values of the coupling strength 
$K>0$, the dynamics of different oscillators in the network are practically uncorrelated. 
For increasing values of $K>0,$ however, the dynamics of the oscillators becomes more and
more synchronized. To describe the degree of synchronization, Kuramoto used
the complex \textit{order parameter}:
\be\lbl{order}
r(t)e^{\1\psi (t)}:=n^{-1}\sum_{j=1}^n e^{\1 \theta_j(t)}.
\ee
Here, $0\le r(t)\le 1$ and $\psi(t)$ stand for the modulus and the argument of the order parameter defined
by the right-hand side of \eqref{order}. Note that if all phases are independent uniform random 
variables and $n\gg 1$ then with probability $1$, $r=o(1)$ by the Strong Law of Large Numbers.
If, on the other hand, all phase variables are equal then $r=1$. Thus, one can interpret the value of 
$r$ as the measure of coherence in the system dynamics. Numerical experiments with the 
KM (with normally distributed frequencies $\omega_i$'s) reveal the phase transition
at a certain critical value of the coupling strength $K_c>0.$ Specifically, numerics suggest that 
for $t\gg 1$ (cf.~\cite{Str00})
$$
r(t)= \left\{ \begin{array}{ll} O(n^{-1/2}),& 0<K<K_c,\\
r_\infty (K)+O(n^{-1/2}), & K>K_c.
\end{array}
\right.
$$
Assuming that $g$ is a smooth even function that is  decreasing  on  $\omega\in \R^+,$ 
Kuramoto derived the formula for the critical value 
\be\lbl{Kc-Kuramoto}
K_c={2\over \pi g(0)}.
\ee
Furthermore, he formally showed that in the partially synchronized regime ($K>K_c$),
the steady-state value for the order parameter is given by 
\begin{equation}
\label{partialS}
r_\infty(K)= \sqrt{\frac{-16}{\pi K_{c}^4 g''(0)}}\sqrt{K - K_{c}} + O(K-K_{c}).
\end{equation}
Recently, Chiba and Nishikawa \cite{ChiNis2011} and
Chiba \cite{ Chi15a} confirmed Kuramoto's heuristic analysis
with the rigorous derivation of \eqref{Kc-Kuramoto} and analyzed the
bifurcation 
at $K_c$.

In this paper, we initiate a mathematical investigation of  the transition to coherence 
in the Kuramoto model on graphs. To this end, we consider the following model:
\be\lbl{wKM-int}
\dot \theta_i=\omega_i +{K\over n} \sum_{j=1}^n W_{nij} \sin (\theta_j-\theta_i),
\ee
$W_n=(W_{nij})$  is  an  $n\times n$ symmetric matrix of weights.
Note that in the classical KM \eqref{classKM}, every oscillator is coupled to every other 
oscillator in the network,
i.e., the graph describing the interactions between the oscillators is the complete graph on 
$n$ nodes. In the modified model \eqref{wKM-int}, we supply the edges of the complete graph
with the weights $(W_{nij})$. Using this framework, we can study 
the KM on a variety of deterministic and random (weighted) graphs.
For instance, let $W_{nij}, \; 1\le i<j\le n$ be independent Bernoulli random variables
$$
\P(W_{nij}=1) =p,
$$
for some $p\in (0,1)$. Complete the definition of $W_n$ by setting $W_{nji}=W_{nij}$ 
and $W_{nii}=0, \, i\in [n].$ With this choice of $W_n,$ \eqref{wKM-int} yields the KM on 
Erd\H{o}s-R{\' e}nyi random graph.  

The family of Erd\H{o}s-R{\' e}nyi graphs parameterized by $n$ is one example 
of a convergent family of random graphs \cite{LovGraphLim12}. The limiting behavior
of such families is determined by a symmetric measurable function on the unit square
$W(x,y)$, called a graphon. In the case of the Erd\H{o}s-R{\' e}nyi graphs, 
the limiting graphon is the constant function $W\equiv p$.
In this paper, we study the KM on convergent families of deterministic and random graphs.
In each case the asymptotic properties of graphs are known through the limiting graphon
$W$. The precise relation between the graphon $W$ and the weight matrix $W_n$ will be 
explained below.

In studies of coupled systems on graphs, one of the main
questions is the relation between the structure of the graph and network dynamics. 
For the problem at hand, this translates into the question
of how the structure of the graph affects the transition to synchrony in the KM.
For the KM on convergent families of graphs, in this paper, we derive the formulas for 
the critical 
values
\be\lbl{wKc}
K_c^+={2\over \pi g(0) \zeta_{max}(\mathcal{W})}\quad\mbox{and}\quad
K_c^-={2\over \pi g(0) \zeta_{min}(\mathcal{W})},
\ee
where $\zeta_{max}(\mathcal{W})\; \left(\zeta_{min}(\mathcal{W}) \right)$ is the largest
positive (smallest negative)  eigenvalue of the self-adjoint kernel operator 
$\mathcal{W}:L^2(I)\to L^2(I), \, I:=[0,1],$ defined by 
\be\lbl{def-W2}
\mathcal{W}[f]=\int_I W(\cdot,y)f(y)dy,\quad f\in L^2(I).
\ee
If all eigenvalues of $\mathcal{W}$ are positive (negative) then 
$K_c^-:=-\infty$ ($K_c^+:=\infty$). 
The main result of this work shows that the incoherent state is linearly (neutrally) stable
for $K\in [K^-_c,K_c^+]$ and is unstable otherwise.
The transition point formulas in \eqref{wKc} reveal the effect of the network topology 
on the synchronization properties of the KM through the extreme eigenvalues 
$\zeta_{max}(\mathcal{W})$ and $\zeta_{min}(\mathcal{W})$. 
For the classical KM ($W\equiv 1$)
$\zeta_{max}(\mathcal{W})=1$ and there are no negative eigenvalues. 
Thus, we recover \eqref{Kc-Kuramoto} from \eqref{wKc}.
 
We derive \eqref{wKc} from the linear stability analysis of the mean field limit of \eqref{wKM-int}.
The latter is a partial differential equation for the probability density function corresponding 
to the distribution of the phase variables on the unit circle (see \eqref{MF}, \eqref{def-V}). 
For the classical KM, the mean field
limit was derived by Strogatz and Mirollo in \cite{StrMir91}.
We derive the mean field limit for the KM on weighted graphs \eqref{wKM-int}
and show that its solutions  approximate probability distribution of the phase variables on finite 
time intervals for $n\gg 1$. Here,  we rely on the theory of Neunzert
developed for the Vlasov equation \cite{Neu78, Neu84} (see also \cite{BraHep77, Dob79, Gol16}), 
which was also  used by  Lancellotti in 
his treatment of the mean field limit for the classical KM \cite{Lan05}. 

With the mean field limit in hand, we proceed to study transition to coherence in \eqref{wKM-int}.
As for the classical  KM, the density of the uniform distribution is a steady state solution of the 
mean field limit. The linear stability analysis in Section~\ref{sec.stability} shows that the density 
of the uniform distribution is neutrally stable for $K_c^-\le K\le
K^+_c$ and is unstable otherwise. Thus, the 
critical values $K^{\pm}_c$ given in \eqref{wKc} mark the loss of
stability of the incoherent state.
The bifurcations at $K^{\pm}_c$ and the formula for the order
parameter corresponding 
to (\ref{partialS})
will be analyzed elsewhere using the techniques from 
\cite{Chi15a, Chi15b}. 

Sections~\ref{sec.approximate} and \ref{sec.examples} deal with applications. In the former
section we collect approximation results, which facilitate application of our results to a wider 
class of models. Further, in Section~\ref{sec.examples}, we discuss the KM for several representative
network topologies : Erd\H{o}s-R{\' e}nyi, small-world, $k$-nearest-neighbor graphs, and
the weighted ring model. We conclude with a brief discussion of our results in
Section~\ref{sec.discuss}.


\section{The mean field limit}\lbl{sec.MF}
\setcounter{equation}{0}
Throughout this paper, we will use  a discretization of $I=[0,1]:$
\be\lbl{Xn}
X_n=\{\xi_{n1}, \xi_{n2},\dots,\xi_{nn}\},\quad \xi_{ni}\in I, \; i\in [n],
\ee
which satisfies the following property
\be\lbl{wXn}
\lim_{n\to\infty} n^{-1}\sum_{i=1}^n f(\xi_{ni})= \int_I f(x)dx, \quad
\forall f\in C(I).
\ee
\begin{ex} 
The following two examples of $X_n$ will be used in constructions of various graphs
throughout this paper.
\begin{enumerate} 
\item The family of sets \eqref{Xn} with $\xi_{ni}=i/n,\; i\in [n]$ satisfies \eqref{wXn}.
\item Let $\xi_1, \xi_2, \dots $ be independent identically distributed (IID) random variables (RVs)
with $\xi_1$ having the uniform distribution on $I$. Then with $X_n=\{\xi_1, \xi_2,\dots, \xi_n\}$ 
\eqref{wXn} holds almost surely (a.s.), by the Strong Law of Large Numbers.
\end{enumerate}
\end{ex}

Let $W$ be a symmetric Lipschitz continuous function on $I^2:$
\be\lbl{Lip-W}
\left| W(x_1,y_1)- W(x_2,y_2)\right|\le L_W \sqrt{ (x_1-x_2)^2 + (y_1-y_2)^2}\quad
\forall (x_{1,2},y_{1,2})\in I^2.
\ee

The weighted graph $\Gamma_n= G(W, X_n)$ on $n$ nodes
is defined as follows. The node and the edge sets of $\Gamma_n$ are
  $V(\Gamma_n)=[n]$ and 
\be\lbl{EGamma}
E(\Gamma_n)=\left\{ \{i,j\}:\;  W(\xi_{ni},\xi_{nj})\neq 0,\; i,j\in [n]\right\},
\ee
respectively. Each edge $\{i,j\}\in E(\Gamma_n)$ is supplied with the weight 
$W_{nij}:=W(\xi_{ni}, \xi_{nj}).$

On $\Gamma_n$, we consider the KM
of phase oscillators
\be\lbl{KM}
\dot \theta_{ni} = \omega_{i} + Kn^{-1} \sum_{j=1}^n W_{nij} \sin(\theta_{nj}-\theta_{ni}),
\quad i\in [n].
\ee
The phase  variable $\theta_{ni}: \R\to\SS:=\R/ 2\pi\Z$  
corresponds to the oscillator at node $i\in [n].$ Throughout this paper,
we identify $\theta\in\SS$ with its value in the fundamental domain, i.e.,
$\theta\in [0, 2\pi)$. Further, we equip $\SS$ with the distance
\be\lbl{S-distance}
d_\SS(\theta,\theta^\prime)=\min \{|\theta-\theta^\prime|, 2\pi-|\theta-\theta^\prime|\}.
\ee

The oscillators at the adjacent nodes
interact through the coupling term on the right hand side of \eqref{KM}.
The intrinsic frequencies $\omega_1, \omega_2, \dots$ are 
IID RVs.
Assume that $\omega_1$
has absolutely continuous probability distribution with a continuous density
$g(\omega)$. 
The initial condition
\be\lbl{KM-ic}
\theta_{ni}(0)=\theta_i^0, \quad i\in [n],
\ee
are sampled independently from 
the conditional probability distributions with 
densities $\hat\rho^0_{\theta|\omega}(\theta, \omega_i, \xi_{ni}), \;
i\in [n]$. Here, $\hat \rho^0_{\theta|\omega}(\theta,\omega,\xi)$ is a nonnegative
continuous function on $G:=\SS\times \R\times I$  that is uniformly
continuous in $\xi$, i.e., $\forall \epsilon >0 \;\exists \delta>0$
such that
\be\lbl{uniform-xi}
  (\xi_1,\xi_2\in I)\;\&\;  |\xi_1-\xi_2|<\delta \;\implies\;
\left|\hat \rho^0_{\theta|\omega}(\theta,\omega,\xi_1)-
\hat \rho^0_{\theta|\omega}(\theta,\omega,\xi_2)\right|
<\epsilon 
\ee
uniformly in $(\theta,\omega)\in \SS\times\R$.
In addition, we assume 
\be\lbl{condP-ic}
\int_\SS \hat\rho^0_{\theta|\omega} (\phi, \omega, \xi)d\phi =1\quad
\forall (\omega, \xi)\in\R\times I.
\ee

We want to show that the dynamics of \eqref{KM} subject to the initial condition 
\eqref{KM-ic} can be described in terms of the probability density function 
$\hat\rho(t,\theta,\omega,x)$ satisfying
the following Vlasov equation 
\be\lbl{MF}
{\p \over \p t} \hat\rho(t,\theta,\omega,x) +{\p \over \p \theta }\left\{ \hat\rho(t,\theta,\omega,x) V(t,\theta,\omega,x)\right\}
=0,
\ee
where
\be\lbl{def-V}
V(t,\theta,\omega,x)=\omega+
K\int_I\int_\R\int_\SS W(x,y) \sin(\phi - \theta ) \hat\rho(t,\phi,\lambda,y) 
d\phi d\lambda d y
\ee
and the initial condition 
\be\lbl{MF-ic}
\hat\rho(0,\theta,\omega,x)=\hat\rho_{\theta|\omega}^0 (\theta,\omega, x)g(\omega).
\ee
By \eqref{condP-ic}, $\hat\rho(0,\theta,\omega,x)$ is a probability
density on $(G,\mathcal{B}(G))$:
\begin{equation}\lbl{check-density}
\begin{split}
\int_\SS\int_\R\int_I \hat\rho(0,\theta,\omega,x) dx d\omega d\theta 
&=\int_\R\int_I\left\{ \int_\SS \hat\rho^0_{\theta|\omega} 
(\theta,\omega, x)d\theta\right\} g(\omega)  dx d\omega \\
&=\int_\R g(\omega)  d\omega =1.
 \end{split}
\end{equation}
Here, $\mathcal{B}(G)$ stands for the Borel $\sigma$-algebra of $G$.

Below, we show that the solutions of the  IVPs for \eqref{KM} and \eqref{MF}, 
generate two families of Borel probability measures parametrized by $t>0$.
To this end, we introduce the following
empirical measure
\be\lbl{EM2}
\mu_t^n(A)=n^{-1} \sum_{i=1}^n \delta_{P_{ni}(t)} (A) \quad \; A\in\mathcal{B}(G),
\ee
where  $P_{ni}(t)=(\theta_{ni}(t), \omega_i, \xi_{ni})\in G.$ 

To compare measures generated by the discrete and continuous systems,
following \cite{Neu84}, we  use the bounded Lipschitz distance:
\be\lbl{Lip-d}
d(\mu,\nu)=\sup_{f\in\mathcal{L}} \left| \int_G f d\mu -\int_G fd\nu\right|,
\quad \mu,\nu\in \mathcal{M},
\ee
where $\mathcal{L}$ is the set of functions
\be\lbl{def-L}
\mathcal{L}=\left\{ f: G\to [0,1]:\; |f(P)-f(Q)|\le  d_G (P,Q),\; P,Q\in G\right\}
\ee
and
$\mathcal{M}$ stands for the space of Borel probability measures on $G$.
Here,
$$
d_G(P,P^\prime)=\sqrt{ d_\SS(\theta,\theta^\prime)^2+(\omega-\omega^\prime)^2+
(x-x^\prime)^2},
$$
for $P=(\theta, \omega, x)$ and $P^\prime=(\theta^\prime, \omega^\prime, x^\prime)$.
The bounded Lipschitz distance metrizes the convergence of Borel probability 
measures on $G$ \cite[Theorem~11.3.3]{Dud02}.

We are now in a position to formulate the main result of this section.
\begin{thm}\lbl{thm.converge}
Suppose $W$ is a Lipschitz continuous function on $I^2$. 
Then  for any $T>0$, there exists a unique weak solution\footnote{See \cite[Remark~1]{Neu78} 
for the definition of the weak solution.}
of the IVP \eqref{MF}, \eqref{def-V}, and \eqref{MF-ic}, $\hat\rho(t,\cdot), \;t\in[0,T]$, 
which provides the density for Borel probability measure on $G$:
\be\lbl{mu}
\mu_t (A) = \int_A \hat\rho(t,P) dP,\quad A\in\mathcal{B}(G), 
\ee
parametrized by $t\in [0,T]$.
Furthermore, 
\be\lbl{pointwise}
d(\mu_t^n,\mu_t)\to 0
\ee
uniformly for $t\in [0,T],$ provided $d(\mu^n_0,\mu_0)\to 0$ as 
$n\to\infty.$ 
\end{thm}

\begin{proof} 
We rewrite \eqref{KM} as follows
\begin{eqnarray}\lbl{extendKM}
\dot{\theta}_{ni} &=& \lambda_{i} + Kn^{-1} \sum_{j=1}^n W(x_{ni},x_{nj}) \sin (\theta_{nj}-
\theta_{ni}),\\
\nonumber
\dot \lambda_{ni}& =& 0,\\
\nonumber
\dot x_{ni} &=& 0,\quad i\in [n],
\end{eqnarray}
subject to the initial condition
\be\lbl{extendKM-ic}
(\theta_{ni}(0), \lambda_{ni}(0), x_{ni}(0))  =  (\theta_{i}^0, \omega_i,  \xi_{ni}), 
\quad i\in [n].
\ee
As before, we consider the empirical measure corresponding to the solutions of 
\eqref{extendKM},
\eqref{extendKM-ic}
\be\lbl{EM}
\mu_t^n(A)=n^{-1} \sum_{i=1}^n \delta_{ P_{ni}(t)} (A), \quad \; A\in\mathcal{B}(G),
\ee
where $P_{ni}(t)=(\theta_{ni}(t), \lambda_{ni}(t), x_{ni}(t))\in G.$

We need to show that $\mu_t^n$ and $\mu_t$
are close for large $n$. This follows from the Neunzert's theory \cite{Neu84}.
Specifically, below we show 
\be\lbl{Dobrushin-I}
d(\mu_t^n,\mu_t)\le C d(\mu^n_0,\mu_0), \quad t\in [0,T],
\ee
for some $C>0$ independent from $n$.

Below, we prove \eqref{Dobrushin-I}.
Theorem~\ref{thm.converge} will then follow.

Let $C(0,T;\mathcal{M})$ denote the space of weakly continuous 
$\mathcal{M}$-valued functions on $[0,T]$. Specifically, $\mu_.\in
C(0,T;\mathcal{M})$
means that 
\be\lbl{weak-cont}
t\mapsto \int_G f(P)d\mu_t(P)
\ee
is a continuous function of $t\in [0,T]$ for every bounded continuous
function $f\in C_b(G)$.

For a given $\nu_.\in C(0,T;\mathcal{M}),$
consider the following equation of characteristics:
\be\lbl{CE}
\frac{dP}{dt}=\tilde V[\nu_{.}](t,P),\quad P(s)=P^0\in G.
\ee
where $P=(\theta,\omega,x)$ and
\be\lbl{VF}
\tilde V[\nu_{.}](t,P)=\begin{pmatrix} \omega+K\int_G W(x,y) \sin(v-\theta)d\nu_{t}(v,\omega,y) \\
0 \\0
\end{pmatrix}.
\ee

Under our assumptions on $W$, \eqref{CE} has a unique global solution,
which depends continuously on initial data.
Thus, \eqref{VF} generates  the flow $T_{t,s}: G\to G$ 
($T_{s,s}=\mathrm{id}, T_{s,t}=T^{-1}_{t,s}$):
$$
P(t)=T_{t,s}[\nu_{.}]P^0.
$$ 
Following \cite{Neu78}, we consider the fixed point equation:
\be\lbl{FP}
\nu_t=\nu_0\circ T_{0,t}[\nu_{.}], \quad t\in [0,T],
\ee
which is interpreted as
$$
\nu_t(A)=\nu_0\left(T_{0,t}[\nu_{.}](A)\right)\quad \forall A\in \mathcal{B}(G).
$$ 
It is shown in \cite{Neu78} that under the conditions \textbf{(I)} and \textbf{(II)}
given below, for any $\nu_0\in \mathcal{M}$ there
is a unique solution of the fixed point equation \eqref{FP} $\nu_.\in C(0,T;\mathcal{M}).$
Moreover, for any two initial conditions $\nu^{1,2}_0\in \mathcal{M},$ we have
\be\lbl{CNT}
\sup_{t\in [0,T]} d(\nu_t^1,\nu_t^2)\le \exp\{CT\} d(\nu_0^1,\nu_0^2)
\ee
for some $C>0$.
By construction of $T_{t,s}$ and \eqref{EM},  the empirical measure $\mu_.^n$ satisfies
the fixed point equation \eqref{FP}. By \cite[Theorem~1]{Neu78}, $\nu_t$,
the solution of the \eqref{FP}, is
an absolutely continuous measure with density $\hat\rho (t,\cdot)$ for every $t\in [0,T]$,  
provided $\nu_0$ is absolutely continuous with density $\hat\rho(0,\cdot)$ 
(cf.~\eqref{check-density}). 
Furthermore, $\hat\rho (t,P)$ is 
a weak solution of the IVP for \eqref{MF}, \eqref{def-V}, and \eqref{MF-ic}.
Therefore, since both the empirical measure $\mu^n_.$ and its continuous counterpart 
$\mu_.$ (cf.~\eqref{EM} and \eqref{mu}) satisfy the fixed point equation \eqref{FP},
 we can use \eqref{CNT} to obtain \eqref{Dobrushin-I}. 
It remains to verify the following two conditions on the vector field $\tilde V[\nu_.]$,
which guarantee the solvability of \eqref{FP} and continuous dependence on initial
data estimate \eqref{CNT} (cf.~\cite{Neu84}):
\begin{description}
\item[(I)] $\tilde V[\mu_{.}](t,P)$ is continuous in $t$ and is globally Lipschitz continuous 
in $P$ with Lipschitz constant\footnote{A straightforward estimate shows
that $L_1=\sqrt{2}\left(L_W+\| W\|_{L^\infty (I^2)}\right)+1$, where $L_W$ is the Lipschitz
constant in \eqref{Lip-W}.} $L_1,$ which depends on $W$.
\item[(II)] The mapping $\tilde V: \mu_{.}\mapsto \tilde V[\mu_{.}]$ is Lipschitz continuous
in the following sense:
$$
 \left|\tilde V[\mu_{.}](t,P)-\tilde V[\nu_{.}](t,P)\right| \le L_2 d (\mu_t, \nu_t),
$$
for some $L_2>0$ and for all $\mu_{.},\nu_{.}\in C(\R,\mathcal{M})$ and 
$(t,P)\in [0,T]\times G$. \footnote{With these assumptions the estimate \eqref{CNT} 
holds with $C=L_1+L_2$ (see~\cite{Neu84}).}
\end{description}

For the Lipschitz continuous function $W$, it is straightforward to verify 
conditions \textbf{(I)} and \textbf{(II)}. In particular, \textbf{(I)} follows
from the weak continuity of $\mu_t$ (cf.~\eqref{weak-cont}) and Lipschitz 
continuity of $W$ and
$\sin x.$ The second condition is verified following the treatment  
of the mechanical system presented in \cite{Neu84} (see also \cite{Lan05}). 
We include the details of the verification of \textbf{(II)} for completeness.

Let $P=(\theta,\omega,x)\in G$ be arbitrary but fixed and define
\be\lbl{def-f}
f(\phi,\lambda,y; P)={W(x,y)\sin (\phi-\theta)+\|W\|_{L^\infty(I^2)} \over 2( \|W\|_{L^\infty(I^2)}+L_W) },
\ee
where  $L_W$ is the Lipschitz constant of $W(x,y)$ (cf. \eqref{Lip-W}). 
Then $f\in \mathcal{L}$ (cf.~\eqref{Lip-d}). Further,
\begin{equation*}
\begin{split}
\left|\tilde V[\nu_{.}](t,P)-\tilde V[\mu_{.}](t,P)\right| 
&= \left|K \int_G W(x,y) \sin (\phi-\theta)
\left( d \nu_t(\phi,\lambda,y)- d \mu_t(\phi,\lambda,y)\right)\right|\\
&= 2K(\|W\|_{L^\infty(I^2)}+L_W) \left| \int_G f(\phi,\lambda,y) 
\left( d \nu_t(\phi,\lambda,y)- d \mu_t(\phi,\lambda,y)\right)\right|\\
&\le  L_2 d(\nu_t,\mu_t), \quad L_2:=2K(\|W\|_{L^\infty(I^2)}+L_W),
\end{split}
\end{equation*}
which verifies the condition (II).
\end{proof}

\begin{cor}\lbl{cor.converge}
For the empirical measure $\mu^n_t$ \eqref{EM2} and absolutely continuous measure 
$\mu_t$ \eqref{mu}
defined on the solutions of the IVPs
\eqref{KM}, \eqref{KM-ic}, \eqref{condP-ic}  and 
\eqref{MF}, \eqref{def-V}, \eqref{MF-ic} respectively,  we have
$$
\lim_{n\to\infty}\sup_{t\in [0,T]} d(\mu^n_t,\mu_t)=0 \quad
\mbox{a.s.}.
$$
\end{cor}
\begin{proof}
In view of Theorem~\ref{thm.converge}, we need to show
$$
\lim_{n\to\infty} d(\mu^n_0,\mu_0)=0 \quad
\mbox{a.s.}.
$$
By \cite[Theorem~11.3.3]{Dud02}, it is sufficient to show
\be\lbl{weak-c}
\lim_{n\to\infty}
\int_G f\left(d\mu_0^n-d\mu_0\right) =0 \quad
\forall f\in BL(G)\quad \mbox{a.s.},
\ee
where $BL(G)$ stands for the space of bounded real-valued Lipschitz functions on $G$
with the supremum norm.
Since $BL(G)$ is a separable space, let $\{f_m\}_{m=1}^\infty$ denote
a dense set in $BL(G)$.
Using \eqref{KM-ic} and \eqref{EM2}, we have
\be\lbl{rewrite-fm}
\int_G f_m d\mu_0^n = n^{-1}\sum_{i=1}^n f_m(\theta^0_i,\omega_i, \xi_{ni})
=:n^{-1}\sum_{i=1}^n Y_{m,ni}.
\ee
RVs $Y_{m,ni}, i\in [n],$ are independent and uniformly bounded.
Further,
\be\lbl{rewrite-f}
\begin{split}
\E\left( n^{-1}\sum_{i=1}^n Y_{m,ni}\right) & = 
n^{-1}\sum_{i=1}^n \int_\SS \int_\R f_m(\phi,\lambda,\xi_{ni}) \hat \rho_{\theta|\omega}^0
(\phi,\lambda,\xi_{ni}) g(\lambda) d\lambda d\phi\\
&=:  n^{-1}\sum_{i=1}^n F_m(\xi_{ni}),
\end{split}
\ee
Because $f_m$ is Lipschitz and $\rho_{\theta|\omega}^0$ is uniformly continuous in $\xi$ 
(cf.~\eqref{uniform-xi}), the function
$$
F_m(\xi)=\int_\SS\int_\R f_m(\phi,\lambda,\xi) 
\hat \rho_{\theta|\omega}^0
(\phi,\lambda,\xi) g(\lambda) d\lambda d\phi
$$
is continuous on $I$.

By \eqref{wXn}, we have 
\be\lbl{Fm}
\lim_{n\to\infty} \E \left(n^{-1}\sum_{i=1}^n Y_{m,ni}\right) =\int_I F_m(\xi)d\xi
=\int_G f_m d\mu_0.
\ee
By the Strong Law of Large Numbers\footnote{ It is easy to adjust the proof
of Theorem~6.1 \cite{Bil-ProbMeas} so that it applies to the triangular array 
$Y_{ni}, i\in [n], n\in \N$ (see \cite[Lemma~3.1]{Med14c}).},
from \eqref{rewrite-fm}, \eqref{rewrite-f}, and \eqref{Fm} we have 
\be\lbl{SLLN}
\lim_{n\to\infty}\int_G f_m \left( d\mu_0^n-d\mu_0\right)=
\lim_{n\to\infty} n^{-1}\sum_{i=1}^n \left(Y_{m,ni}-\E Y_{m,ni}\right)=
0 \quad \mbox{a.s.}.
\ee
Therefore,
$$
\P\left\{ \lim_{n\to\infty}\int_G f_m  \left( d\mu_0^n-d\mu_0\right)= 0 \;
\forall m\in\N\right\}=1.
$$
Using  density of $\{f_m\}_{m=1}^\infty$ in $BL(G)$, we have
$$
\P\left\{ \lim_{n\to\infty}\int_G f  \left( d\mu_0^n-d\mu_0\right)= 0 \;
\forall  f\in  BL(G)\right\}=1.
$$
\end{proof}

\section{Linear stability}\lbl{sec.stability}
\setcounter{equation}{0}
In the previous section, we established that the Vlasov equation
\eqref{MF}, approximates discrete system \eqref{KM} for $n\gg 1$. Next, we will use \eqref{MF}
to characterize the transition to synchrony for increasing $K$. To this end, in this section,
we derive the linearized equation about the incoherent state, the steady state solution of the 
mean field equation. In the next section, we will study how the spectrum of the linearized
equation changes with $K$.

In this section, we assume that probability 
density $g$ is a continuous and even function monotonically decreasing  on $\R^+$.

\subsection{Linearization}

First, setting $\hat\rho(t,\theta,\omega,x)=\rho(t,\theta,\omega,x) g(\omega)$, 
from \eqref{MF} we derive the equation for $\rho$:
\be\lbl{MMF}
{\p\over\p t} \rho +{\p\over \p\theta} \left\{ V_\rho \rho\right\}=0,
\ee
where
$$
V_\rho(t,\theta,\omega,x)=\omega +
K \int_I\int_\SS\int_\R W(x,y)\sin(\phi-\theta)\rho(t,\phi,\lambda,y)g(\lambda)d\lambda d\phi dy.
$$

By integrating \eqref{MF} over $\SS$, one can see that 
$$
{\p \over \p t} \int_\SS \hat\rho(t,\theta,\omega,x) d\theta=0,
$$
and, thus,
\be\lbl{equal-g}
\int_\SS \hat\rho(t,\theta,\omega,x) d\theta=
\int_\SS \hat\rho(0,\theta,\omega, x) d\theta
= \int_\SS \hat \rho^0_{\theta|\omega}(0,\theta,\omega, x) g(\omega)d\theta
=g(\omega).
\ee

Thus,
\be\lbl{balance}
\int_\SS \rho(t,\theta,\omega,x) d\theta =1\quad \forall (t,\omega,x)\in \R^+\times\R\times I.
\ee
In addition,
$$
\int_\R g(\omega) d\omega =1,
$$
because $g$ is a probability density function.
The density of the uniform distribution on $\SS$, $\rho_u=(2\pi)^{-1},$ is a steady-state solution
of the mean field equation \eqref{MMF}. It corresponds to the completely mixed state. We are interested
in stability of this solution. In the remainder of this section, we derive the linearized equation
around $\rho_u$.

Let 
\be\lbl{perturb-rho}
\rho=\rho_u+z(t,\theta, \omega, x).
\ee
By \eqref{balance},
\be\lbl{zero-mean}
\int_\SS z(t,\theta, \omega, x )d\theta =0 \quad \forall (t, \omega, x )\in\R^+\times \R\times I.
\ee

By plugging in \eqref{perturb-rho} into \eqref{MMF}, we obtain
\be\lbl{Eq-z-1}
{\p\over \p t} z(t,\theta, \omega, x ) +{\p\over \p\theta} \left\{ V_{\rho_u+z}(t,\theta, \omega, x ) 
\left({1\over 2\pi} +z\right)\right\} =0.
\ee
The expression in the curly brackets has the following form:
\begin{equation*}
\begin{split}
V_{\rho_u+z}(t,\theta, \omega, x ) \left({1\over 2\pi} +z \right)&=
\left(\omega+K 
\int_I\int_\SS\int_\R W(x,y)\sin(\phi-\theta) \times \right.\\
&\times\left. \left((2\pi)^{-1}+z(t, \phi,\lambda,y)\right) g(\lambda)d\lambda d\phi dy \right)
\left({1\over 2\pi} +z\right)= {\omega\over 2\pi} 
 + \omega z \\
&+ {K\over 2\pi}  
\int_I\int_\SS\int_\R W(x,y)\sin(\phi-\theta) z(t,\phi,\lambda,y)g(\lambda)d\lambda d\phi dy
+O(z^2).
\end{split}
\end{equation*}
Thus,
$$
{\p\over \p t} z +{\p\over \p \theta} \left\{ \omega z +{K\over 2\pi}\mathcal{G}[z] \right\}
+O(z^2) =0,
$$
where the linear operator $\mathcal{G}$ is defined by
\be\lbl{def-P2}
\mathcal{G}[z]:=
\int_I\int_\SS\int_\R W(x,y)\sin(\phi-\theta) z(t,\phi,\lambda,y)g(\lambda)d\lambda d\phi dy
\ee

We arrive at the linearized equation:
\be\lbl{linearized}
{\p\over \p t} Z +{\p\over \p \theta} \left\{ \omega Z +{K\over 2\pi}\mathcal{G}[Z] \right\}=0.
\ee

\subsection{Fourier transform}
We expand $Z$ into Fourier series
\be\lbl{Fourier}
Z(t, \theta,\omega,x)= \sum_{k=1}^\infty \hat{Z}_k(t,\omega,x)e^{-\1 k\theta}+
\overline{\left(\sum_{k=1}^\infty \hat{Z}_k(t,\omega,x)e^{-\1 k\theta}\right)},
\ee
where $\hat{Z}_k$ stands for the Fourier transform of $Z$
$$
\hat{Z}_k={1\over 2\pi} \int_\SS Z(\theta,\cdot) e^{\1 k\theta} d\theta.
$$
In \eqref{Fourier}, we are using the fact that $Z$ is real and 
$\hat{Z}_0=0$ (cf.~\eqref{zero-mean}).

The linear stability of $\rho_u$ is, thus, determined by the time-asymptotic
behavior of $\hat Z_k,\; k\ge 1$. To derive the differential equations for 
$\hat Z_k,\; k\ge 1$, we apply the Fourier transform to \eqref{linearized}:
\be\lbl{F-transf}
{\p\over \p t} \hat{Z}_k +
\widehat{\left({\p\over \p \theta} \left\{ \omega Z +{K\over 2\pi}\mathcal{G}[Z] \right\}
\right)_k}=0.
\ee

Using the definition of the Fourier transform and integration by parts, we have
\be\lbl{2nd-term}
\widehat{\left({\p\over \p \theta} \left\{ \omega Z +{K\over 2\pi}\mathcal{G}[Z] \right\}
\right)_k}=
\frac{1}{2\pi}\int_\SS {\p\over \p\theta}\left(\dots\right) e^{\1k\theta}d\theta=
-\1 k \omega \hat{Z}_k-\1 k {K\over 2\pi}\widehat{\left(\mathcal{G}[Z]\right)_k}.
 \ee
It remains to compute $\widehat{\left(\mathcal{G}[Z]\right)_k},\; k\ge 1.$
To this, end we rewrite
\begin{equation}\lbl{Pz}
\begin{split}
\mathcal{G}[Z]&={1\over 2\1} \int_I\int_\SS\int_\R W(x,y) 
\left(e^{\1(\phi-\theta)}-e^{-\1(\phi-\theta)}\right)
Z(t,\phi,\lambda,y)g(\lambda)d\lambda d\phi dy\\
&={\pi\over \1} \int_I \int_\R W(x,y) \left(e^{-\1\theta}\hat{Z}_1 -e^{\1\theta}
\hat{Z}_{-1}\right) g(\lambda)d\lambda dy.
\end{split}
\end{equation}

Using \eqref{Pz}, we compute
\begin{equation}\lbl{FT-2nd}
\begin{split}
 \widehat{\left(\mathcal{G}[Z]\right)_k} &= \frac{1}{2\pi}{\pi\over \1} 
\int_\SS \int_I \int_\R W(x,y) \left(e^{\1(k-1)\theta}\hat{Z}_1 -e^{\1(k+1)\theta}
\hat{Z}_{-1}\right) g(\lambda)d\lambda dy d\theta\\
&=\left\{ \begin{array}{ll}
{\pi\over \1}\mathcal{P}[\hat{Z}_1], & k=1,\\
0,& k>1,
\end{array}
\right.
\end{split}
\end{equation}
where
\be\lbl{def-P}
\mathcal{P}[Z]:=\int_I\int_\R W(x,y) Z(t,\lambda,y) g(\lambda) d\lambda dy.
\ee 

The combination of \eqref{F-transf}, \eqref{2nd-term}, and \eqref{FT-2nd} yields
the system of equations for $\hat{Z}_k,\; k\ge 1$:
\begin{eqnarray}\lbl{Zk-1}
{\p \over \p t} \hat{Z}_1 &=&\1 \omega \hat{Z}_1 +{K\over 2} \mathcal{P}[\hat{Z}_1],\\
\lbl{Zk-2}
{\p \over \p t} \hat{Z}_k &=&\1 k\omega \hat{Z}_k,\quad k>1.
\end{eqnarray}

\subsection{Spectral analysis}
In this section, we study \eqref{Zk-1}, which decides the linear stability
of the mixed state. We rewrite \eqref{Zk-1} as 
\be\lbl{reZk-1}
{\p\over \p t} \hat{Z}_1(t,\omega,x)=T[\hat{Z}_1](t,\omega,x).
\ee
where 
\be\lbl{def-T}
T[Z]=\1\omega Z+{K\over 2} \mathcal{P}[Z].
\ee

Equations \eqref{def-P} and \eqref{def-T} define linear operators
$\mathcal{P}$ and $T$ on the weighted Lebesgue space 
$L^2(X,gd\omega dx)$ with $X:=\R\times I$.

\begin{lem}\lbl{lem.spT}
$T:L^2(X,gd\omega dx)\to L^2(X,gd\omega dx)$ is a closed operator. The residual
spectrum of $T$ is empty and the continuous spectrum
$\sigma_c(T)=\1 \operatorname{supp}(g)$.
\end{lem}
\begin{proof}
Consider the multiplication operator 
$\mathcal{M}_{\1\omega}:L^2(X,gd\omega dx)\to L^2(X,gd\omega dx)$ defined by
\be\lbl{def-M}
\mathcal{M}_{\1\omega} z=\1\omega z, \quad \omega \in \R.
\ee
It is well known that $\mathcal{M}_{\1\omega}$ is closed and
the (continuous) spectrum of $\mathcal{M}_{\1\omega}$ lies on the imaginary axis
$\sigma_c(\mathcal{M}_{\1\omega})=\1 \cdot \mathrm{supp}(g)$.
Since $W(x,y)$ is square-integrable by the assumption,
$\mathcal{P}$ is a Hilbert-Schmidt operator on $L^2(X,gd\omega dx)$ and,
therefore, is compact \cite{Young-Hilbert}.  
Then, the statement of the lemma follows from  the perturbation theory for linear operators \cite{Kato}.
\end{proof}

Similarly, the spectrum of the operator $\mathcal{M}_{\1 j\omega}$ lies on the imaginary axis;
$\sigma (\mathcal{M}_{\1j\omega}) = ij\cdot \mathrm{supp}(g)$.
Hence, the trivial solution $\hat{Z}_j \equiv 0$ of (\ref{Zk-2}) for $j=2,3,\dots$ is neutrally stable.

We define a Fredholm integral operator $\mathcal{W}$ on $L^2(I)$ by
\be\lbl{def-W}
\mathcal{W}[V](x)=\int_\R W(x,y) V(y) dy.
\ee
Since $\mathcal{W}$ is compact, the set of eigenvalues
$\sigma_p(\mathcal{W})$ is a bounded countable with the only accumulation point at
the origin.
Since $\mathcal{W}$ is symmetric, all eigenvalues are real numbers.

\begin{lem}\lbl{lem.EV-T}
The eigenvalues of $T$ are given by
\be\lbl{sigmaT}
\sigma_p (T)=\left\{ \lambda\in\C \backslash \sigma_c(T) 
:\; D(\lambda) ={2\over \zeta K},\; \zeta\in\sigma_p (\mathcal{W})\backslash \{0\}\right\},
\ee
where
\be\lbl{D}
D(\lambda ):= \int_{\R}\! \frac{1}{\lambda -i\omega }g(\omega )d\omega . 
\ee
\end{lem}
\begin{proof}
Suppose $v\in L^2(X,g d\omega dx)$ is an eigenvector of $T$ corresponding to the 
eigenvalue $\lambda$:
$$
T[v]=\lambda v.
$$
Then
\be\lbl{eqn-v}
v=2^{-1}K(\lambda-\1\omega)^{-1} \mathcal{P}[v].
\ee
By multiplying both sides of \eqref{eqn-v} by $g(\omega)$ and integrating
with respect to $\omega,$ we have
\be\lbl{eigenvalue-equation}
\int_{\R}\! v(\omega,x)g(\omega)d\omega 
=\frac{K}{2}\int_{\R}\! \frac{1}{\lambda -i\omega }g(\omega)d\omega \cdot
\int_{I}\int_{\R}\! W(x,y)v(\omega ,y)g(\omega )d\omega dy.    
\ee
Rewrite \eqref{eigenvalue-equation} as
\begin{equation}\lbl{EV-W}
V ={K\over 2} D(\lambda)  \mathcal{W}[V],
\end{equation}
where
$$
V:=\int_\R v(\omega,\cdot) g(\omega) d\omega \in L^2(I).
$$

Equations \eqref{EV-W} and \eqref{def-W} reduce the eigenvalue problem for $T$ to that for the 
Fredholm operator $\mathcal{W}$.
Suppose $V \in L^2(I)$ is an eigenfunction of $\mathcal{W}$ associated with the eigenvalue $\zeta \neq 0$.
Then  (\ref{EV-W}) implies $D(\lambda ) = 2/(\zeta K)$.

If $0 \in \sigma _p(\mathcal{W})$ and $V$ is a corresponding  eigenvector, then
Equation \eqref{eqn-v} yields
\begin{eqnarray*}
v&=& \frac{K}{2}\frac{1}{\lambda -i\omega }\int_{I}\!\int_{\R}\! W(x,y) v(\omega ,y) g(\omega )d\omega dy \\
&=& \frac{K}{2}\frac{1}{\lambda -i\omega } \int_{I}\! W(x,y) V(y) dy
 = \frac{K}{2}\frac{1}{\lambda -i\omega } (\mathcal{W}[V])(x) = 0.  
\end{eqnarray*}
Thus, $\zeta = 0$ is not an eigenvalue of $T$.
\end{proof}

For $\zeta\neq 0$ denote
\be\lbl{Kzeta}
K(\zeta) = \frac{2}{\pi g(0)}\frac{1}{|\zeta|}.
\ee
\begin{lem} \lbl{lem.critical}
For each $\zeta \in \sigma _p(\mathcal{W})$ and $K>K(\zeta)$ there is a unique eigenvalue
of $T$ $\lambda =\lambda (\zeta, K)$ satisfying $D(\lambda ) = 2/(\zeta K)$.

For $\zeta\in \sigma _p(\mathcal{W})\bigcap \R^+,$  $\lambda(\zeta, K)$ is a positive increasing 
function of $K$ satisfying
\be\lbl{zeta+}
\lim_{K \to K(\zeta)+0} \lambda (\zeta, K) = 0+0, \quad \lim_{K \to \infty} \lambda (\zeta, K) = \infty.
\ee

For $\zeta\in \sigma _p(\mathcal{W})\bigcap \R^-,$  $\lambda(\zeta, K)$ is a negative decreasing 
function of $K$ satisfying
\be\lbl{zeta-}
\lim_{K \to K(\zeta)+0} \lambda (\zeta, K) = 0-0, \quad \lim_{K \to \infty} \lambda (\zeta, K) = -\infty.
\ee

Finally,
$$
\sigma_p(T)=\left\{ \lambda(\zeta, K):\quad \zeta\in\sigma_p(\mathcal{W})\backslash \{ 0\}, \; K>K(\zeta)\right\}
 \subset \R \backslash \{ 0\}.
$$
\end{lem}
\begin{proof}
Since $\zeta \in \R$, setting 
$\lambda = x+\1 y$ for the equation $D(\lambda ) = 2/(\zeta K)$ yields
\begin{eqnarray}
\left\{ \begin{array}{l}
\displaystyle \int_{\R}\! \frac{x}{x^2 + (\omega -y)^2} g(\omega )d\omega = \frac{2}{\zeta K},   \\
\displaystyle \int_{\R}\! \frac{\omega -y}{x^2 + (\omega -y)^2} g(\omega )d\omega = 0. \\
\end{array} \right.
\label{eigeneq2}
\end{eqnarray}
With $\zeta=1$ this system of equations was analyzed in \cite{Chi15a} in the context of
the classical Kuramoto model \eqref{classKM}.

 The second equation of (\ref{eigeneq2}) gives
\be\lbl{eqn-im}
0 = \int_{\R}\! \frac{\omega -y}{x^2 + (\omega -y)^2} g(\omega ) d\omega 
    = \int^{\infty}_{0}\! \frac{\omega }{x^2 + \omega ^2} (g(y+\omega ) - g(y-\omega )) d\omega .
\ee
Since $g$ is even, $y=0$ is a solution of \eqref{eqn-im}.
Furthermore, since  $g$ is unimodal, there are no other solutions of \eqref{eqn-im}.
Thus, $\lambda$ is real.

With $y=0$ the first equation of (\ref{eigeneq2}) yields
\be\lbl{eqn-re}
\int_{\R}\! \frac{x}{x^2 + \omega ^2} g(\omega )d\omega = \frac{2}{\zeta K}.
\ee
Since $g$ is nonnegative and $K>0$, from \eqref{eqn-re} we have $\zeta x>0$.
Further, the left-hand side of \eqref{eqn-re} satisfies
\begin{eqnarray}\lbl{Poisson}
\lim_{x \to \pm 0}\int_{\R}\! \frac{x}{x^2 + \omega ^2} g(\omega )d\omega &=& \pm \pi g(0),\\
\lbl{odd}
\lim_{x \to \pm \infty }\int_{\R}\! \frac{x}{x^2 + \omega ^2} g(\omega )d\omega &=&0.
\end{eqnarray}
The first identity  follows from the Poisson's integral formula for the upper half-plane 
\cite{SteinWeiss}. The combination of \eqref{eqn-re} and \eqref{Poisson} 
implies that $x\to 0$ as $K\to K(\zeta) + 0$, while  that of \eqref{eqn-re} and \eqref{odd} 
yields $x\to \pm \infty$ as $K\to \infty$.

For the uniqueness, it is sufficient to show that the function
$$
x\mapsto \int_{\R}\! \frac{x}{x^2 + \omega ^2} g(\omega )d\omega
$$
is monotonically decreasing in $x$ except at the jump point $x=0$.
If this were not true, there would exist two eigenvalues for some interval $K_1 < K < K_2$, 
and two eigenvalues would collide and disappear at $K = K_1$ or $K=K_2$.
This is impossible, because $D(\lambda )$ is holomorphic in $\lambda $. 
\end{proof}

To formulate the main result of this section, we will need the
following notation:
\be\lbl{def-xi}
\begin{split}
\xi_{max}(\mathcal{W})=&\left\{\begin{array}{ll}
\max\{ \zeta: \quad 
\zeta\in\sigma_p(\mathcal{W})\bigcap \R^+\},\;
\sigma_p(\mathcal{W})\bigcap \R^+\neq\emptyset,\\
0+0,\; \mbox{otherwise}.
\end{array}
\right.\\
\xi_{min}(\mathcal{W})=&\left\{\begin{array}{ll}
\min\{ \zeta: \quad 
\zeta\in\sigma_p(\mathcal{W})\bigcap \R^-\},\;
\sigma_p(\mathcal{W})\bigcap \R^-\neq\emptyset,\\
0-0,\; \mbox{otherwise}.
\end{array}
\right.
\end{split}
\ee
Further, let
\be\lbl{Kc}
K_c^+={2\over \pi g(0) \xi_{max}(\mathcal{W})}\quad\mbox{and}\quad
K_c^-={2\over \pi g(0) \xi_{min}(\mathcal{W})}.
\ee

\begin{thm}\lbl{thm.critical}
The spectrum of $T$ consists of the continuous spectrum on the imaginary 
axis and possibly one or more negative eigenvalues, if  $K\in[K_c^-, K_c^+]$,
and there is at least one positive eigenvalue of $T$, otherwise.
\end{thm}

Theorem~\ref{thm.critical} follows from Lemma~\ref{lem.critical}. 
It shows that the  incoherent state is linearly (neutrally) 
stable for $K\in[K_c^-, K_c^+]$ and is unstable otherwise. The critical values $K^{\pm}_c$ 
mark the loss of stability the incoherent state. 
The detailed
analysis of the bifurcations at $K^\pm_c$ will be presented elsewhere.
\begin{rem}\lbl{rem.compare}
For the classical Kuramoto model on the complete graph, $W\equiv 1$ 
and $\zeta_{max} (\mathcal{W})=1$. Thus, we recover the well-known
Kuramoto's transition formula \eqref{Kc-Kuramoto} \cite{Kur75}. 
In the general case, the transition points depend on  the graph 
structure through the extreme eigenvalues of the kernel operator 
$\mathcal{W}$. 
\end{rem}

\begin{rem}
For nonnegative graphons $W$, $\zeta_{max}(\mathcal{W})$ coincides with the
spectral radius of  the limiting kernel operator $\mathcal{W}$ (cf.~\cite{Sze2011}):
$$ 
\varrho (\mathcal{W})=\max\{ |\zeta|:\quad \zeta\in\sigma_p(\mathcal{W})\},
$$ 
This can be seen from the variational characterization of the eigenvalues of a
self-adjoint compact operator, which also implies
\be\lbl{Fischer}
\varrho (\mathcal{W})=\max_{\| f\|_{L^2(I)}=1} (\mathcal{W}[f], f).
\ee
\end{rem}


\section{Approximation}\lbl{sec.approximate}
\setcounter{equation}{0}

Equation \eqref{KM} may be viewed as a base model.
To apply our results to a wider class of deterministic and random
networks, we will need approximation results, which are collected in this section.

\subsection{Deterministic networks}

Consider the Kuramoto model on the weighted graph 
$\tilde\Gamma_n=\langle [n], E(\tilde\Gamma_n),\tilde W \rangle$
\be\lbl{wKM}
\dot{\tilde\theta}_{ni} = \omega_{i} + Kn^{-1} \sum_{j=1}^n
\tilde W_{nij} \sin(\tilde\theta_{nj}-\tilde\theta_{ni}),
\quad i\in [n],
\ee
where $\tilde W_n=(\tilde W_{nij})$ is a symmetric matrix.

Denote the corresponding empirical measure by
\be\lbl{tEM}
\tilde\mu_t^n(A)=n^{-1} \sum_{i=1}^n \delta_{\tilde P_{ni}(t)} (A), \quad \; A\in\mathcal{B}(G),
\ee
where $\tilde P_{ni}(t)=(\tilde\theta_{ni}(t), \omega_i, \xi_{ni}), \; i\in[n].$

First, we show that if $W_n$ and $\tilde W_n$ are close, so are the solutions of the IVPs for \eqref{KM} and
\eqref{wKM} with the same initial conditions.  
To measure the proximity of $W_n=(W_{nij})\in \R^{n\times n}$ and $\tilde W_n=(\tilde W_{nij})$
and the corresponding solutions $\theta_n$ and $\tilde \theta_n$, we will use the following norms:
\be\lbl{norms}
\|W_n\|_{2,n}=\sqrt{n^{-2}\sum_{i,j=1}^n W_{nij}^2}, \quad
\|\theta_n \|_{1,n}=\sqrt{n^{-1}\sum_{i=1}^n \theta_{ni}^2}, 
\ee
where $\theta_n=(\theta_{n1},\theta_{n2},\dots,\theta_{nn})$ and $W_n=(W_{nij})$.

The following lemma will be used to extend our results for the KM \eqref{KM}
to other networks.
\begin{lem}\lbl{lem.cont-depW}
Let $\theta_n(t)$ and $\tilde\theta_n(t)$ denote solutions of the IVP for
\eqref{KM} and \eqref{wKM} respectively. Suppose that
the initial data for these problems coincide
\be\lbl{bdd-ic} 
\theta_n(0)=\tilde\theta_n(0).
\ee 
Then for any $T>0$ there exists $C=C(T)>0$ such that  
\be\lbl{ave} 
\max_{t\in [0,T]} \left\|\theta_n(t)-\tilde\theta_n(t)\right\|_{1,n} \le C 
\left\|W_n-\tilde W_n\right\|_{2,n}, 
\ee
where the positive constant $C$ is independent from $n$.
\end{lem}
\begin{cor}\lbl{cor.cont-dep-mu}
\be\lbl{cdep-mu}
\sup_{t\in [0,T]} d(\mu_t^n,\tilde\mu_t^n)\le  C \left\|W_n-\tilde W_n\right\|_{2,n}.
\ee
\end{cor}

\begin{proof}(Lemma~\ref{lem.cont-depW})
Denote $\phi_{ni}=\theta_{ni}-\tilde\theta_{ni}$.
By subtracting (\ref{wKM}) from (\ref{KM}), multiplying the result by
$n^{-1}\phi_{ni},$ and summing over $i\in [n]$, we obtain
\begin{eqnarray}\nonumber
(2K)^{-1} {d\over dt} \|\phi_n\|^2_{1,n} &=& n^{-2}\sum_{i,j=1}^n
\left(W_{nij}-\tilde W_{nij}\right)
\sin(\theta_{nj}-\theta_{ni}) \phi_{ni}\\
\nonumber
&+& 
n^{-2} \sum_{i,j=1}^n \tilde W_{nij}\left(\sin(\theta_{nj}-\theta_{ni})-
\sin(\tilde\theta_{nj}-\tilde\theta_{ni})\right)\phi_{in}\\
\lbl{subtract}
&=:&I_1+I_2.
\end{eqnarray}

Using an obvious bound $\left|\sin(\theta_{nj}-\theta_{ni})\right|\le 1$ and
an elementary inequality $|ab|\le 2^{-1}( a^2+b^2)$, we obtain
\be\lbl{I1}
|I_1|\le 2^{-1} \left(\|W_n-\tilde W_{n}\|_{2,n}^2+\|\phi_n\|_{1,n}^2\right).
\ee
Further, from Lipschitz continuity of $\sin$ and the definition of $\phi_{ni}$, we have
$$
\left|\sin(\theta_{ni}-\theta_{nj})-
\sin(\tilde\theta_{ni}-\tilde\theta_{nj})\right| \le \left| \phi_{ni}-\phi_{nj}\right|\le
| \phi_{ni}|+|\phi_{nj}|.
$$
Therefore, 
\be\lbl{I2}
|I_2|\le  2 \|\tilde W\|_{L^\infty (I^2)} \|\phi_n\|^2_{1,n}.
\ee
The combination of \eqref{subtract}, \eqref{I1}, and \eqref{I2} yields
\be\lbl{pre-Gron}
K^{-1}{d\over dt} \|\phi_n(t)\|_{1,n}^2 \le 
\left( 4\| \tilde W\|_{L^\infty (I^2)} +1\right) \|\phi_n(t)\|_{1,n}^2 +\|W_n-\tilde W_n\|_{2,n}^2.
\ee
By the Gronwall's inequality \cite[Appendix B.2.j]{Eva-PDE}, 
\begin{eqnarray*}
\|\phi_n(t)\|_{1,n}^2 &\le & e^{KC_1t}\left( \|\phi_n(0)\|_{1,n}^2 + 
\int_0^t K\|W_n-\tilde W_n\|_{2,n}^2  ds \right)\\
&\le & e^{KC_1t}Kt  \|W_n-\tilde W_n\|_{2,n}^2,
\end{eqnarray*}
where we used $\phi_n(0)=0$ and $C_1:=\left( 4\| \tilde W\|_{L^\infty (I^2)} +1\right)$.
Thus,
$$
\sup_{t\in [0,T]} \|\phi_n(t)\|_{1,n}^2\le e^{KC_1T} KT \|W_n-\tilde W_n\|_{2,n}^2.
$$
\end{proof}

\begin{proof}(Corollary~\ref{cor.cont-dep-mu})
Let $f\in\mathcal{L}$ (cf.~\eqref{def-L}) and consider
\begin{equation*}
\begin{split}
\left|\int_G f \left( d\mu_t^n-d\tilde\mu_t^n \right)\right| &=
\left|n^{-1} \sum_{i=1}^n f(\theta_{ni}(t))-f(\tilde\theta_{ni}(t))\right| \\
&\le n^{-1}\sum_{i=1}^n \left|\theta_{ni}(t)-\tilde\theta_{ni}(t)\right| \\
&\le \|\theta_n(t)-\tilde\theta_n(t)\|_{1,n}.
\end{split}
\end{equation*}
By Lemma~\ref{lem.cont-depW},
$$
\max_{t\in [0,T]} d(\mu^n_t,\tilde\mu^n_t) 
=\sup_{f\in\mathcal{L}} \left|\int_G f\left( d\mu_t^n-d\tilde\mu_t^n\right)\right|
\le C\|W_n-\tilde W_n\|_{2,n}.
$$
\end{proof}

\subsection{Random networks}
We now turn to the KM on random graphs. To this end, we use 
W-random graph $\bar \Gamma_n=G_r(X_n,W)$, which we define next. As before, 
$X_n$ is a set of points \eqref{Xn},\eqref{wXn}. $\bar\Gamma_n$ is 
a graph on $n$ nodes, i.e., $V(\bar\Gamma_n)=[n]$. The edge set is defined as follows:
\be\lbl{Pedge}
\P\left( \{i,j\}\in E(\Gamma_n)\right)= W(\xi_{ni}, \xi_{nj}).
\ee
The decision for each pair $\{i,j\}$ is made independently from the decisions on other pairs.

The KM on the W-random graph $\bar\Gamma_n=G_r(X_n,W)$ has the following form:
\be\lbl{xiKM}
\dot{\bar\theta}_{ni} = \omega_{i} + Kn^{-1} \sum_{j=1}^n e_{nij} \sin(\bar\theta_{nj}-\bar\theta_{ni}),
\quad i\in [n]:=\{1,2,\dots,n\}
\ee
where  $e_{nij}, 1\le i\le j\le n$ are independent Bernoulli RVs:
$$
\P(e_{nij}=1)=W(\xi_{ni},\xi_{nj}),
$$
and $e_{nij}=e_{nji}.$ 
\begin{lem}\lbl{lem.ave}
Let $\theta_n(t)$ and $\bar\theta_n(t)$ denote solutions of the IVP for
\eqref{KM} and \eqref{xiKM} respectively. Suppose that
the initial data for these problems coincide
\be\lbl{Wbdd-ic} 
\theta_n(0)=\bar\theta_n(0).
\ee 
Then 
\be\lbl{Wave}
\lim_{n\to\infty} 
\sup_{t\in [0,T]} \left\|\theta_n(t)-\bar\theta_n(t)\right\|_{1,n} =0 \quad \mbox{a.s.}.
\ee
\end{lem}

\begin{proof} 
The proof follows the lines of the proof of Lemma~\ref{lem.cont-depW}.
As before, we set up the equation for $\phi_{ni}:=\bar\theta_{ni}- \theta_{ni}$:
\begin{eqnarray}\nonumber
(2K)^{-1} {d\over dt} \|\phi_n\|^2_{1,n} &=& n^{-2}\sum_{i,j=1}^n \left(e_{nij}-W_{nij}\right)
\sin(\theta_{nj}-\theta_{ni}) \phi_{ni}\\
\nonumber
&+& 
n^{-2} \sum_{i,j=1}^n e_{nij}\left(\sin(\bar\theta_{nj}-\bar\theta_{ni})-
\sin(\theta_{nj}-\theta_{ni})\right)\phi_{in}\\
\lbl{subtract-1}
&=:&I_1+I_2.
\end{eqnarray}

As in \eqref{I2}, we have
bound
\be\lbl{est-I_1}
|I_2|\le {n}^{-2} \sum_{i,j=1}^n \left( | \phi_{ni}|+|\phi_{nj}|\right) | \phi_{ni}|\le 
2 \|\phi_n\|^2_{1,n},
\ee
where we used $0\le e_{nij}\le 1.$ 

Next, we turn to the first term on the right hand side of \eqref{subtract-1}.
For this, we will need the following definitions:
\begin{eqnarray*}
Z_{ni}(t)   &=& n^{-1}\sum_{j=1}^n a_{nij}(t)\eta_{nij},\\
a_{nij}(t) &=& \sin\left(\theta_{nj}(t)-\theta_{ni}(t)\right),\\
\eta_{nij} &=&e_{nij}-W_{nij}.
\end{eqnarray*}
and $Z_n=(Z_{n1}, Z_{n2},\dots,Z_{nn})$.
With these definitions in hand, we estimate $I_1$ as follows:
\be\lbl{est-I-1}
|I_1|=|n^{-1}\sum_{i=1}^n Z_{ni}\phi_{ni}|\le 
 2^{-1}\left( \|Z_n\|_{1,n}^2 +\|\phi_n\|^2_{1,n}\right),
\ee

The combination of \eqref{subtract-1}-\eqref{est-I-1} yields,
\be\lbl{pre-Gron2}
{d\over dt} \|\phi_n(t)\|_{1,n}^2 \le 5K\|\phi_n(t)\|_{1,n}^2 +K\|Z_n(t)\|_{1,n}^2.
\ee
Using the Gronwall's inequality and \eqref{Wbdd-ic}, we have 
$$
\|\phi_n(t)\|_{1,n}^2\le K \exp\{ 5Kt\} \int_0^t \|Z_n(s)\|_{1,n}^2 ds
$$
and 
\be\lbl{Gron}
\sup_{t\in [0,T]} \|\phi_n(t)\|_{1,n}^2\le K\exp\{5KT\} \int_0^T \|Z_n(t)\|_{1,n}^2 dt.
\ee

 Our next goal is to estimate $\int_0^T \|Z_n(t)\|_{1,n}^2 dt$. Below, we show that
\be\lbl{our-goal}
 \int_0^T\|Z_n(t)\|_{1,n}^2 dt\to 0\quad\mbox{as}\; n\to\infty \;\mbox{a.s.}.
 \ee
To this end, we will use the following observations.
Note that $\eta_{nik}$ and $\eta_{nil}$ are independent for $k\neq l$
and
\be\lbl{mean-0}
\E \eta_{nij}=  \E e_{nij}-W_{nij} =0,
\ee
where we used $\P (e_{nij}=1) = W_{nij}$.

Further,
\be\lbl{eta2}
\begin{split}
\E\eta_{nij}^2 &= \E (e_{nij}-W_{nij})^2 = \E e_{nij}^2- W_{nij}^2 \\
&=  W_{nij}-W_{nij}^2 \le 1/4.
\end{split}
\ee
and
\be\lbl{eta4}
\begin{split}
\E(\eta_{nij}^4) & = \E  (e_{nij}-W_{nij})^4\\
& = \E\left(e_{nij}^4-4e_{nij}^3W_{nij} +6 e_{nij}^2 W_{nij}^2-4e_{nij} W_{nij}^3
+W_{nij}^4\right)\\
&= W_{nij}-4W_{nij}^2+6W_{nij}^3-3W_{nij}^4\\
&= W_{nij}\left(1-W_{nij}\right)^4+ W_{nij}^4\left(1-W_{nij}\right) \le 2^{-4}.
\end{split}
\ee

Next, 
\be\lbl{Zni2}
\int_0^T Z_{ni}(t)^2 dt= n^{-2} \sum_{k,l=1}^n c_{nikl} \eta_{nik}\eta_{nil}, 
\ee
where
\be\lbl{def-cnjikl}
c_{nikl}=\int_0^T a_{nik}(t) a_{nil}(t) dt \quad\mbox{and}\quad |c_{nikl}|\le T. 
\ee
Further,
\be\lbl{Zni2-2}
\int_0^T \|Z_{n}(t)\|_{1,n}^2 dt= n^{-3} \sum_{i,k,l=1}^n c_{nikl} \eta_{nik}\eta_{nil}
\ee
and, finally,
\be\lbl{Zni4}
\E \left( \int_0^T \|Z_n(t)\|^2_{1,n}dt\right)^2 =
n^{-6} \sum_{i,k,l, j,p,q=1}^n c_{nikl} c_{njpq} 
\E\left(\eta_{nik}\eta_{nil} \eta_{njp}\eta_{njq}\right).
\ee
We have six summation indices $i, k, l, j, p, q$ ranging from $1$ to $n$. Since $\E\eta_{nik}=0$
for $i,k \in [n],$ and RVs $\eta_{nik}$ and $\eta_{njp}$ are independent whenever 
$\{i,k\}\neq \{j,p\}$, the nonzero terms on the right-hand side of \eqref{Zni4} fall into two
groups:
\begin{center}
\begin{tabular}{c c}
$I \; :$ & \quad $c_{nikk}^2 \eta_{nik}^4$ \\
$II\; :$ &\quad  $c_{nikk}  c_{njpp} \eta_{nik}^2\eta_{njp}^2\; (i\neq j)$\quad
or $\quad  c_{nikl}^2 \eta_{nik}^2\eta_{nil}^2\; (k\neq l).$
\end{tabular}
\end{center}

There are $n^2$ terms of type $I$ and $3n^3(n-1)$ terms of type $II$.
Thus,
\be\lbl{2nd-moment}
\E \left( \int_0^T \|Z_n(t)\|^2_{1,n}dt\right)^2 \le T^2 n^{-6} \left( n^2+3n^3(n-1)\right)
= O(n^{-2}),
\ee
where we used \eqref{Zni4}, \eqref{eta2}, \eqref{eta4}, and the 
bound on $|c_{nikl}|$ in \eqref{def-cnjikl}.

Next, for a given $\epsilon>0$ we denote
\be\lbl{def-An}
A_n^\epsilon=\left\{ \left|\int_0^T \|Z_n(t)\|_{1,n}^2 dt\right| 
\ge \epsilon \right\}.
\ee
and use Markov's inequality and \eqref{2nd-moment} to obtain
\be\lbl{union-An}
\sum_{n=1}^\infty \P (A_n^\epsilon) \le \epsilon^{-2} \sum_{n=1}^\infty
\E \left( \int_0^T \|Z_n(t)\|^2_{1,n}dt\right)^2<\infty.
\ee
By the Borel-Cantelli Lemma,  \eqref{def-An} and \eqref{union-An} imply 
\eqref{our-goal}. The latter combined with \eqref{Gron} concludes the proof 
of Lemma~\ref{lem.ave}.
\end{proof}

\section{Examples}\lbl{sec.examples}
\setcounter{equation}{0}
\subsection{Erd\H{o}s-R{\' e}nyi graphs}\lbl{sec.ER}
Let $p\in (0,1)$, $X_n$ be defined in \eqref{Xn}, \eqref{wXn}, and
$W_p(x,y)\equiv p$. Then $\Gamma_{n,p}=G_r(X_n,W_p)$ is a family of 
Erd\H{o}s-R{\' e}nyi
random graphs. To apply the transition point formula 
\eqref{Kc}, we need to compute the largest eigenvalue of the self-adjoint compact operator 
$\mathcal{W}_p: L^2(I)\to L^2(I)$
defined by 
\be\lbl{Wp}
\mathcal{W}_p[f](x)=\int_I W(x,y)f(y) dy=p\int_I f(y)dy, \quad f\in L^2(I).
\ee

\begin{lem}\lbl{lem.ER}
The largest eigenvalue of $\mathcal{W}_p$ is $\zeta_{max}(\mathcal{W}_p) = p$.
\end{lem}
\begin{proof}
Suppose $\lambda\in\R\backslash \{0\}$ is an eigenvalue of 
$\mathcal{W}_p$ and $v\in L^2(I)$ is the corresponding
eigenvector. Then
\be\lbl{EV-identity}
p\int_I v(y)dy= \zeta v(x).
\ee
Since the right-hand side is not identically $0$, 
$v=\mbox{constant}\neq 0$. By integrating both sides
of \eqref{EV-identity}, we find that $\zeta=p$.
\end{proof}

Thus, for the KM on Erd\H{o}s-R{\' e}nyi graphs, we have
$$
K^+_c={2\over \pi g(0)p}, \quad K_c^- = -\infty.
$$

\subsection{Small-world network} \lbl{sec.SW}

Small-world (SW) graphs interpolate between regular nearest-neighbor graphs and
completely random Erd\H{o}s-R{\' e}nyi graphs. They found
widespread applications, because they combine  features of regular symmetric graphs 
and random graphs, just as seen in many real-world networks \cite{WatStr98}.

Following \cite{Med14b, Med14c}, we construct SW graphs as W-random graphs 
\cite{LovSze06}.
To this end let $X_n$ be a set of points from \eqref{Xn} satisfying \eqref{wXn},
and define $W_{p,r}:I^2\to I$ by
\be\lbl{def-Wpr}
W_{p,r}(x,y)=
\left\{ 
\begin{array}{ll}
1-p, & d_\SS(2\pi x,2\pi y)\le 2\pi r,\\
p, &\mbox{otherwise},
\end{array}
\right.
\ee
where $p,r\in (0,1/2)$ are two parameters.

\begin{df}\lbl{df.Wpr} \cite{Med14b}
$\Gamma_n=G_r(X_n, W_{p,r})$ is called W-small-world graph.
\end{df}

The justification of the mean field limit in Section~\ref{sec.MF}
relies on the assumption that the graphon $W$ is Lipschitz continuous.
To apply Theorem~\ref{thm.converge} to the model at hand, we approximate 
the piecewise constant graphon $W_{p,r}$ by 
a Lipschitz continuous $W_{p,r}^\epsilon$:
\be\lbl{L2close}
\|W_{p,r}-W^\epsilon_{p,r}\|_{L^2(I^2)}<\epsilon.
\ee
Further, we approximate the SW graph $\Gamma_n=G_r(X_n, W_{p,r})$ by 
$\Gamma_n^\epsilon=G_r(X_n, W^\epsilon_{p,r})$.
The approximation results in Section~\ref{sec.approximate} guarantee that the 
empirical measures generated by the solutions of the IVPs for the KMs on 
$\Gamma_n$ and $\Gamma_n^\epsilon$  (with the same initial data) are
$O(\epsilon)$ close with high probability for sufficiently large $n$. Thus, below
we derive the transition point formula for the KM on $(\Gamma^\epsilon_n)$ and 
take the limit as $\epsilon\to 0$ to obtain the critical values for the 
KM on SW graphs $(\Gamma_n)$.

Lemma~\ref{lem.ave} and Theorem~\ref{thm.converge} justify \eqref{MF}, \eqref{def-V}
as the continuum limit for the KM on the sequence of SW graphs $(\Gamma_n^\epsilon)$.
Thus, \eqref{Kc} yields the transition point formula for the KM on SW graphs.
To use this formula, we need to compute the extreme eigenvalues of 
$\mathcal{W}_{p,r}^\epsilon: L^2(I)\to L^2(I)$ defined by
$$
\mathcal{W}_{p,r}^\epsilon [f]=\int_I W^\epsilon_{p,r} (\cdot, y) f(y) dy, \quad f\in L^2(I).
$$

\begin{lem}\lbl{lem.SW}
The largest eigenvalue of $\mathcal{W}_{p,r}^\epsilon$ is 
\be\lbl{radius-Wpr}
\zeta_{max} (\mathcal{W}_{p,r}^\epsilon) =2r +p -4pr +o_\epsilon(1).
\ee
\end{lem}

\begin{proof}
We calculate the largest eigenvalue of $\mathcal{W}_{p,r}$.
Using the definition \eqref{def-Wpr}, one can write
$$
\mathcal{W}_{p,r} [f](x)= \int_\SS K_{p,r} (x-y) f(y) dy,
$$
where $K_{p,r}$ is a $1-$periodic function on $\R,$ whose restriction 
to the interval $[-1/2, 1/2)$ is defined as follows
\be\lbl{def-Kpr}
K_{p,r}(x)=
\left\{ 
\begin{array}{ll}
1-p, & |x|\le r,\\
p, &\mbox{otherwise},
\end{array}
\right.
\ee

The eigenvalue problem for $\mathcal{W}_{p,r}$ can be rewritten as
$$
K_{p,r} \ast v=\zeta v.
$$ 

Thus, the eigenvalues of $\mathcal{W}_{p,r}$ are given by the Fourier coefficients of $K_{p,r}(x)$ as
$$
\zeta_k= (\hat{K}_{p,r} )_k := \int_{\SS}\! K_{p,r}(x)e^{-2\pi i kx}dx, 
$$
for $k\in \Z$.
The corresponding eigenvectors $v_k= e^{\1 2\pi kx}, k\in \Z,$ form a complete
orthonormal set in $L^2(I)$.

A straightforward calculation yields
\be\lbl{zeta-k}
\zeta_k =\left\{ \begin{array}{ll} 2r+p-4rp, & k=0,\\
(\pi k)^{-1} (1-2p) \sin (2\pi kr), & k\neq 0.
\end{array}
\right.
\ee
Further,
\be\lbl{specR-SW}
\zeta_{max}(\mathcal{W}_{p,r})=2r+p-4rp.
\ee
The estimate \eqref{radius-Wpr} follows from \eqref{specR-SW} and \eqref{L2close}
via continuous dependence of the eigenvalues with respect to the parameter $\varepsilon $.
\end{proof}

Thus, for the KM on SW graphs, we obtain
\begin{eqnarray*}
K^+_c = \frac{2}{\pi g(0)}\frac{1}{2r+p-4pr}.
\end{eqnarray*}
The smallest negative eigenvalue is given by $\zeta_{k(r)}$ for some $k(r) \neq 0$.
While it is difficult to obtain an explicit expression $k(r)$, from \eqref{zeta-k}
 we can find a lower bound for it:
\be\lbl{low}
\xi_{min}(\mathcal{W}) > {-1+2p\over \pi}.
\ee
In particular, the incoherent state is stable for 
$K\in \left(2\left( g(0)(-1+2p)\right )^{-1}, 2\left( g(0)\pi (2r+p-4pr)\right)^{-1}\right]$.

\subsection{Coupled oscillators on a ring}

A common in applications type of network connectivity can be described as follows. Consider 
$n$ oscillators placed uniformly around a circle. They are labelled by integers
from $1$ to $n$ in the counterclockwise direction. To each potential edge $\{i,j\}\in [n]^2$
we assign a weight
\be\lbl{isotropic}
W_{nij}=G(\xi_{ni}-\xi_{nj}),
\ee
where $G$ is a $1$-periodic even measurable bounded function.

\begin{ex} \lbl{ex.k-neighbor}
Let $X_n=\{1/n,2/n,\dots, 1\}$ and the restriction of the $1$-periodic
even function $G_r$ on $[0,1/2]$ is defined by 
\be\lbl{Kr}
G_r(x) =\left\{\begin{array}{ll} 1, & 0\le x\le r,\\
0, &x>r,
\end{array}
\right.
\ee
where $r\in (0,1/2)$ is a parameter. With this choice of $G:=G_r$, we obtain a 
$k$-nearest-neighbor model, in which each node is connected to $k=\lfloor rn\rfloor$
from each side.
\end{ex}
 
\begin{ex}\lbl{ex.KB}
Another representative example was used by Kuramoto and Battogtokh \cite{KurBat02}.
Here, let $X_n=\{1/n,2/n,\dots, 1\}$ and the restriction 
of the $1$-periodic
even function $G$ to $[0,1/2]$ is defined by 
$G(x):=e^{-\kappa x}$, where $\kappa>0$
is a parameter. With this choice of $G$, we obtain the KM where the strength of interactions
decreases exponentially with the distance between oscillators.
\end{ex}

As in our treatment of the KM on SW graphs in \S\ref{sec.SW}, the integral operator
$\mathcal{W}: L^2(I)\to L^2(I)$ can be written as a convolution
\be\lbl{conv}
\mathcal{W}[f](x)=\int_I G(x-y) f(y) dy,\quad f\in L^2(I).
\ee
It is easy to verify that the eigenvalues of $\mathcal{W}$ are given by 
the Fourier coefficients of $G(x)$.

For instance, for the $k$-nearest-neighbor network
in Example~\ref{ex.k-neighbor}, by setting $p=0$ in \eqref{specR-SW}, for the network
at hand we obtain
$
\zeta_{max}(\mathcal{W})=r
$
and, thus,
$$
K^+_c={1\over \pi g(0) r}.
$$
Note that, in accord with our physical intuition, the transition point is inversely 
proportional to  the coupling range.

For the network in Example~\ref{ex.KB}, eigenvalues are given by
\begin{equation}
\zeta_k := \left\{ \begin{array}{ll}
\displaystyle \frac{2\kappa}{\kappa^2 + 4\pi^2 k^2} (1-e^{-\kappa/2}), & 
(k: \text{even}), \\
\displaystyle \frac{2\kappa}{\kappa^2 + 4\pi^2 k^2} (1+e^{-\kappa/2}), & 
(k: \text{odd}),\quad k=0,1,2,\dots.
\end{array} \right.
\end{equation}
The largest positive eigenvalue is $\zeta_{max} (\mathcal{W}) = \zeta_0$, and we obtain
$$
K^+_c = \frac{\kappa}{\pi g(0)}\frac{1}{1-e^{-\kappa /2}}, \quad K_c^- = -\infty.
$$
This recovers the classical result (\ref{Kc-Kuramoto}) as $\kappa \to 0$.

If the explicit expression for the largest positive eigenvalue of $\mathcal{W}$
is not available, for a network with nonnegative
graphon $W$, the transition point can be estimated using the variational 
characterization of the largest eigenvalue \eqref{Fischer}. Specifically,
from \eqref{Fischer}, we have:
$$
\zeta_{max} (\mathcal{W}) \ge \int_{I^2} W dxdy =\|W\|_{L^1 (I^2)},
$$
and, thus,
$$
K^+_c\le {2\over \pi g(0) \| W \|_{L^1(I^2)}}.
$$

\section{Discussion}\lbl{sec.discuss}
\setcounter{equation}{0}

In this work, we derived and rigorously justified the mean field equation for the 
KM on convergent families of graphs. Our theory covers a large class of coupled 
systems. In particular, it clarifies the mathematical meaning of the mean field 
equation used in the analysis of chimera states (see, e.g., \cite{Ome13}). Moreover,
we show how to write the mean field equation for the KM on many common in 
applications random graphs including Erd\H{o}s-R{\' e}nyi and small-world 
graphs, for which it has not been known before.

We used the  mean field equation to study synchronization in the KM on large deterministic and
random graphs. We derived the transition point formulas for the critical values of the coupling
strength, at which the incoherent state looses stability. The transition point formulas 
show explicit dependence of the stability boundaries of the incoherent state on the 
spectral properties of the limiting graphon. This reveals the precise mechanism by which the 
network topology affects the stability of the incoherent state and the onset of synchronization.
In the follow-up work, we will show that the linear stability analysis of this paper can be extended
to show nonlinear stability of the incoherent state albeit with respect to the weak topology.
There we will also present the bifurcation analysis for the critical values $K^\pm_c$.
The analysis of the KM on small-world networks shows that, unlike in the original KM 
\eqref{classKM}, on graphs the incoherent state may remain stable even for negative 
values of $K$, i.e., for repulsive coupling. In fact, the bifurcations at the left and right
endpoints can be qualitatively different. In the small-world case, the center manifold 
at $K^-_c$ is two-dimensional, whereas it is one-dimensional at $K_c^+$. These 
first findings indicate that the bifurcation structure of the KM on graphs \eqref{KM} 
is richer than that of its classical counterpart \eqref{classKM} and motivates further
investigations of this interesting problem. In the future, we also plan to extend our 
analysis to the KM on certain sparse graphs, including sparse power law networks 
considered in \cite{KVMed16}.

\vskip 0.2cm
\noindent
{\bf Acknowledgements.}
This work was supported in part by the NSF DMS 1412066 (GM).

\vfill
\newpage


 \vfill\newpage
 \bibliographystyle{amsplain}
 \bibliography{knet1}

\end{document}